\documentclass[12pt]{article}
\usepackage{graphicx}
\usepackage{amsmath}
\usepackage{amsfonts}
\usepackage{amssymb}
\usepackage{color}
\usepackage{float}
\providecommand{\U}[1]{\protect\rule{.1in}{.1in}}
\newtheorem{theorem}{Theorem}

\newtheorem{corollary}[theorem]{Corollary}

\newtheorem{definition}[theorem]{Definition}

\newtheorem{lemma}[theorem]{Lemma}

\newtheorem{proposition}[theorem]{Proposition}

\newenvironment{proof}[1][Proof]{\noindent\textbf{#1.} }{\ \rule{0.5em}{0.5em}}
\def\myblacksquare{\rule{1.2ex}{1.2ex}}

\begin{document}

$\ $

\vspace{2.cm}

\begin{center}
{\Large \textbf{Extremal Cylinder Configurations II: \\ \vskip .2cm Configuration $O_6$}}

\vspace{.4cm} {\large \textbf{Oleg Ogievetsky$^{\diamond\,\ast}$\footnote{Also
at Lebedev Institute, Moscow, Russia.} and Senya Shlosman$^{^{\diamond}
\,\dag\,\ddagger}$}}

\vskip .3cm $^{\diamond}$Aix Marseille Universit\'{e}, Universit\'{e} de
Toulon, CNRS, \\ CPT UMR 7332, 13288, Marseille, France

\vskip .05cm $^{\dag}$Inst. of the Information Transmission Problems, RAS,
Moscow, Russia

\vskip .05cm $^{\ddagger}$ Skolkovo Institute of Science and Technology,
Moscow, Russia

\vskip .05cm $^{\ast}${Kazan Federal University, Kremlevskaya 17, Kazan
420008, Russia}
\end{center}

\vskip .4cm
\hspace{5.64cm} {\sf I have a silly walk and I'd like to obtain} 

\vskip .2cm
\hfill{\sf a Government grant to help me develop it.}

\vskip .3cm
\hfill {\it Monty Python} $\ \ \ \ \ \ \ \ $

$\ $

\begin{abstract}\noindent
We study the octahedral configurations $O_6$ \cite{K2} of six equal cylinders touching the unit sphere. We show that the configuration $O_6$ is a local sharp maximum of
the distance function. Thus it is not unlockable and, moreover, rigid. 
\end{abstract}

\newpage
\tableofcontents
\vspace{-.4cm}

\section{Introduction}
In the present paper we continue to study \textit{critical configurations} of six infinite nonintersecting right circular  
cylinders touching the unit sphere. We call a cylinder configuration \textit{critical} if for each small deformation $t$ 
that keeps the radii of the cylinders, either
\begin{itemize}
\item[(T1)] some cylinders start to intersect, or else
\item[(T2)] the distances between all of them increase, but by no more than
$$\sim\left\Vert t\right\Vert ^{2}\ ,$$
or stay zero, for some. The norm 
$\left\Vert t\right\Vert $ is defined by formula $\left( \ref{norm}\right)  $ below. 
\end{itemize}
\noindent A critical configuration is called a \textsl{locally maximal} configuration if all its deformations are of the first type. Any other critical configuration  is called a
\textsl{saddle configuration}, and the deformations of type (T2) are then called the \textsl{unlocking} deformations.
This definition of a critical configuration is close in spirit to the definition of a critical point of the Morse function, but is adapted to our case of the
function being non-smooth minimax function, compare with Definition 4.6 in \cite{KKLS}.

\vskip.1cm
For example, let $C_{6}$ be the configuration of six nonintersecting cylinders of radius $1,$ parallel to the
$z$ direction in $\mathbb{R}^{3}$ and touching the unit ball centered at the origin. One of the results of \cite{OS} is that the configuration $C_{6}$
is a saddle point configuration:  there is a deformation of $C_{6}$ along which the unit cylinders cease touching each other; thus, the configuration $C_{6}$ can be unlocked.
We note here that the structure of the critical point $C_{6}$ is  complicated; in particular, the distance function $D$ (the minimum of the distances between the cylinders) is not even continuous at
$C_{6},$ and the limits $\lim_{\mathbf{m}\rightarrow C_{6}}D\left( \mathbf{m}\right)  ,$ $\mathbf{m}\in M^{6}$, depend on the direction from which
the point $C_{6}$ is approached. Here $M^{6}$ is the relevant configuration space, see the precise definition in Section \ref{sectConfManifold}. 

\vskip .1cm
We have constructed in \cite{OS} the deformation $C_{6,x}$ of the configuration $C_6$. Moving along $C_{6,x}$ the common radius of cylinders grow  
when $x$ decreases from $1$ to $1/2$ ($x=1$ corresponds to the initial configuration $C_6$). For $x=1/2$ 
we obtain the configuration $C_{\mathfrak{m}}$, see Figure \ref{confCm}, for which the radius reaches its maximum value $\frac{1}{8}\left(  3+\sqrt{33}\right)$.
\begin{figure}[H]
\centering
\includegraphics[scale=0.24]{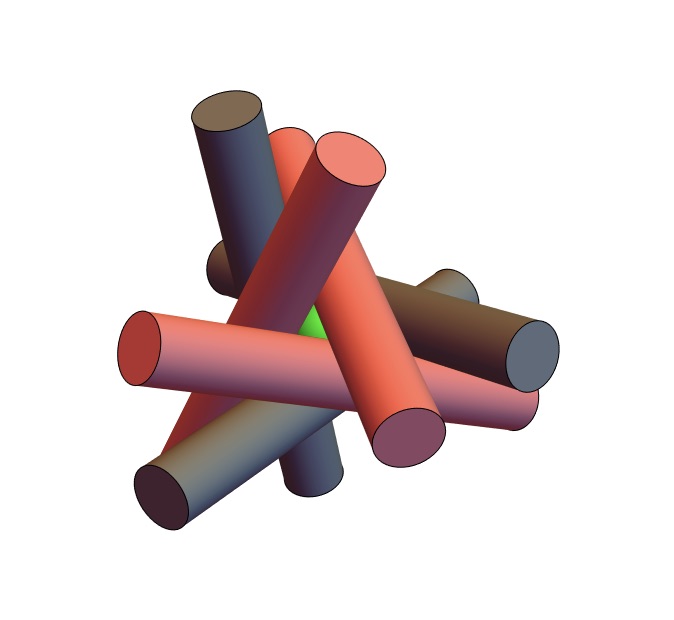}
\vspace{-.4cm} 
\caption{Configuration $C_{\mathfrak{m}}$} 
\label{confCm}
\end{figure}

\vskip .1cm
In \cite{OS-C6} we have shown that the configuration $C_{\mathfrak{m}}$ is a sharp local maximum of the distance function.

\vskip .1cm
In the present paper we study the configuration $O_{6},$ comprised of the
following six radius one cylinders:
\vspace{-.2cm}
\begin{itemize}
\item two cylinders are parallel to the $Oz$ axis and touch the sphere $\mathbb{S}^{2}$ at points $\left(\pm 1,0,0\right)  $ on the $Ox$ axis;
\vspace{-.3cm}
\item two cylinders are parallel to the $Ox$ axis and touch the sphere $\mathbb{S}^{2}$ at points $\left( 0,\pm 1,0\right)  $ on the $Oy$ axis;
\vspace{-.3cm}
\item two cylinders are parallel to the $Oy$ axis and  touch the sphere $\mathbb{S}^{2}$ at points $\left( 0,0,\pm 1\right)  $ on the $Oz$ axis.
\end{itemize}
The letter `O' in the name of the configuration refers, probably, to the fact that the points at which the cylinders touch the sphere
form the vertices of the regular octahedron. In a forthcoming publication \cite{OS-PC} we will give an interpretation of the configuration 
$O_6$ which rather relates it to the regular tetrahedron and suggest a generalization for dual pairs of Platonic bodies. 

\vskip .1cm
The configuration $O_6$ is centrally symmetric. There is a freedom in the definition of the configuration $O_6$:
two cylinders, touching the sphere $\mathbb{S}^{2}$ at points $\left(\pm 1,0,0\right)  $ are parallel to the $z$-axis. 
Instead one can start with the two cylinders, touching the sphere $\mathbb{S}^{2}$ at points $\left(\pm 1,0,0\right)  $ but parallel to the $y$-axis;  then add
the two cylinders, touching the sphere $\mathbb{S}^{2}$ at points $\left(0,\pm 1,0\right)  $ but parallel to the $z$-axis and 
the two cylinders, touching the sphere $\mathbb{S}^{2}$ at points $\left(0,0,\pm 1\right)  $ but parallel to the $x$-axis.
However this configuration is obtained from $O_6$ by the rotation around an arbitrary coordinate axis through the angle $\pi/2$ or $-\pi/2$. 

\vskip .1cm
The configuration $O_6$ of cylinders is shown on Figure \ref{octahedrConfCyl} (the green unit ball is in the center).
\begin{figure}[H]
\vspace{-.8cm} \centering
\includegraphics[scale=0.5]{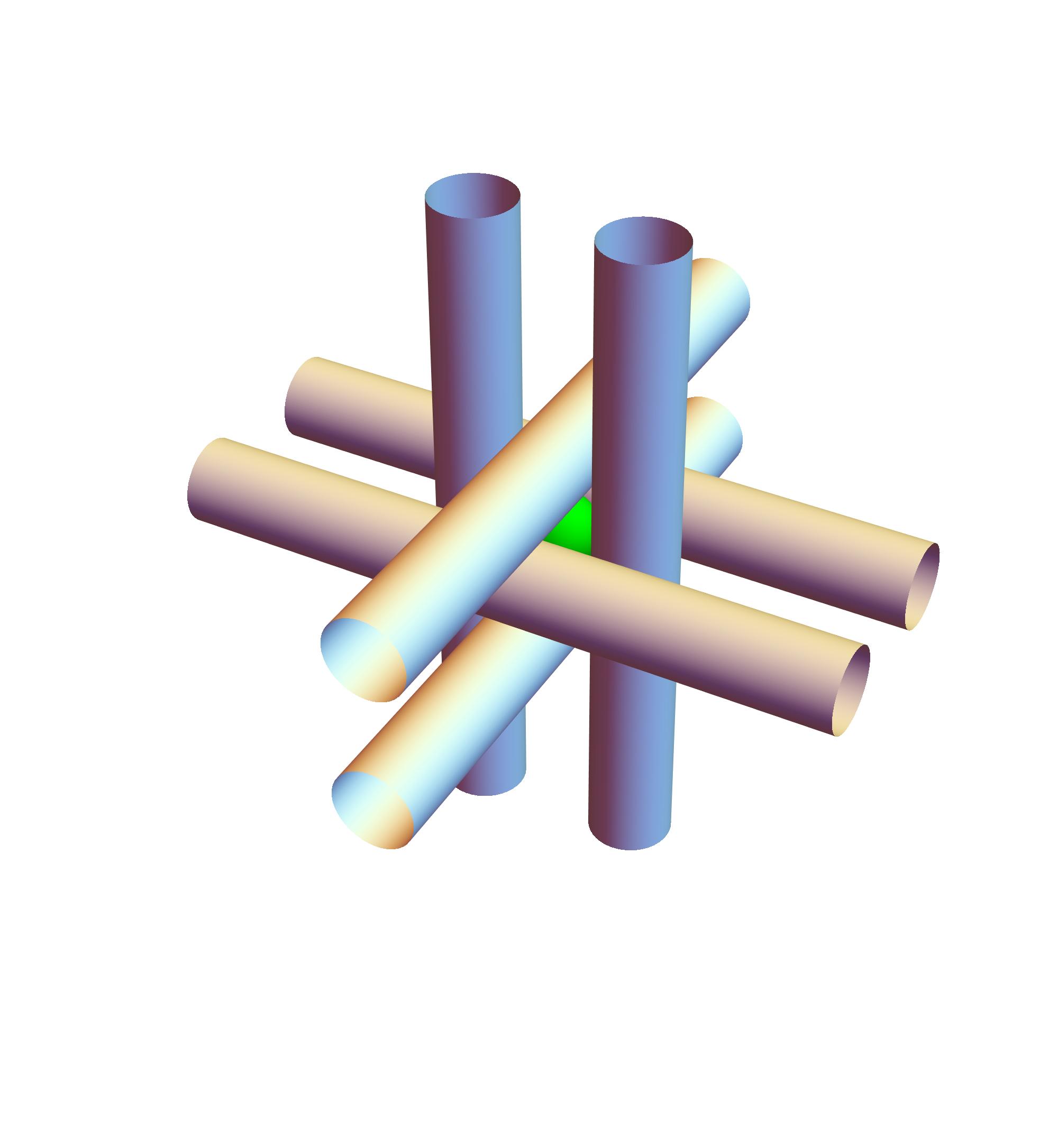}
\vspace{-1cm} 
\caption{Configuration $O_6$ of cylinders}
\label{octahedrConfCyl}
\end{figure}

\vskip .1cm
To visualize configurations of cylinders it is convenient to replace each cylinder by its unique generator (a line
parallel to the axis of the cylinder) touching the sphere $\mathbb{S}^{2}$. We define the value of the distance function on a configuration to be the minimum of distances between 
these tangent lines.

\vskip .1cm
The configuration $O_6$ of tangent lines is shown on Figure \ref{octahedrConf}.
\begin{figure}[th]
\vspace{-.4cm} \centering
\includegraphics[scale=0.7]{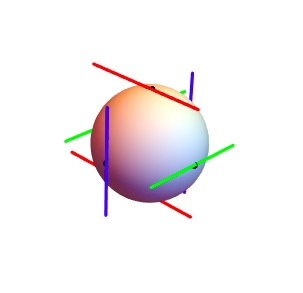}
\vspace{-1cm} 
\caption{Configuration $O_6$ of tangent lines}
\label{octahedrConf}
\end{figure}

\vskip .1cm
\noindent About the configuration $O_6$ W. Kuperberg was asking \cite{K} whether it can be unlocked, i.e. whether one can deform it in such a way that all the distances between the cylinders become positive.
Our result is that the configuration $O_{6}$ is a sharp local maximum of the distance function $D$, and therefore is not unlockable. Moreover, it is rigid, that is, any continuous deformation, which does not increase the radii of cylinders, reduces to a global rotation in the three-dimensional space. 

\vskip .1cm
In the process of the proof we will, in particular, show that the 15-dimensional
tangent space to $M^{6}\,\operatorname{mod}SO\left(  3\right)  $ at $O_{6}$ contains a 6-dimensional subspace along which
the function $D\left(  \mathbf{m}\right)$ decays quadratically, while along any other tangent direction it decays linearly.

\vskip .1cm
As for the configuration $C_{\mathfrak{m}}$ it turns out that it is sufficient to study the variations of distances up to the second order.

\vskip .1cm
For the configuration $O_6$ we distinguish twelve distances between the cylinders which are not parallel. Let 
$\widetilde{D}\left(  \mathbf{m}\right)$ be the minimum of these twelve distances. 
We prove that the configuration $O_{6}$ is a sharp local
maximum already of the function $\widetilde{D}$. 

\vskip .1cm
We first show that there are three convex linear dependencies $\lambda_a$, $a=1,2,3$, between the differentials of the twelve distances.  
We thus have a six-dimensional linear subspace $E$ of the tangent space on which all twelve differentials vanish.  

\vskip .1cm
It so happens that in our coordinates the groups of coordinates entering these linear combinations are disjoint.

\vskip .1cm
For the configuration $C_{\mathfrak{m}}$ there is one convex linear dependency between the differentials, see \cite{OS-C6} and the restriction of same linear combination $Q$ 
of the second differentials on the subspace, on which the differential vanish, is negatively defined. We have shown in \cite{OS-C6} that these conditions are sufficient for the local
maximality. In the present paper we prove a generalization (of the above result for the configuration $C_{\mathfrak{m}}$) which allows us to make a conclusion about the local 
maximality of the configuration $O_6$. 
In this modification the negativity of the form $Q$ is replaced by the non-existence of non-trivial solutions of the system of three 
inequalities $Q_a>0$, $a=1,2,3$, where $Q_a=\lambda_a(Q_1,\dots,Q_{12})\vert_E$. If there existed a convex linear combination of the forms  $Q_a>0$, $a=1,2,3$, 
with negatively defined restriction on $E$, we could simply refer to the assertion made in \cite{OS-C6}. However, this is not the case, see Section \ref{threeforms}.

\vskip .1cm
Let $\mathcal{Z}$ be the configuration space of six unit cylinders touching a unit ball. At his mathoverflow page \cite{K2} W. Kuperberg asks, among other questions:
\begin{itemize}
\item Is the space $\mathcal{Z}$  connected?
\item In particular, within $\mathcal{Z}$, is a continuous transition between the configurations $C_6$ and $O_6$ possible?
\end{itemize}

In the process of our proof of the local maximality of the configuration $O_6$, we establish that the configuration $O_6$ is an isolated point in the space  $\mathcal{Z}$ mod $SO(3)$ which implies the negative answer to these questions, see
Corollary \ref{cornonconne}, Subsection \ref{configO6}. So a modified question arises:
\begin{itemize}
\item How many components does the space $\mathcal{Z}$ have? 
\end{itemize}

The configuration $C_{\mathfrak{m}}$, in contrast to the configuration $O_6$,  is not mirror symmetric. 
We make several conjectures concerning the components of the space  $\mathcal{Z}$ mod $O(3)$ and mod $SO(3)$.

\vskip .1cm
The paper is organized as follows. In the next section we recall 
notation, concerning the manifold $M^{6}$, formulate the maximality result and discuss the connected components of the space $\mathcal{Z}$. Section \ref{maxconfO6} contains the
calculation of the necessary differentials. In Section \ref{genforms} we establish analytic results needed for the proofs of the local maximality of the configuration $O_6$.  

\section{Preliminaries}
A cylinder $\varsigma$ touching the unit sphere $\mathbb{S}^{2}$ has a unique generator (a line
parallel to the axis of the cylinder) $\iota(\varsigma)$ touching $\mathbb{S}^{2}$. We will usually represent a configuration $\{\varsigma
_{1},\dots,\varsigma_{L}\}$ of cylinders touching the unit sphere by the configuration $\{\iota(\varsigma_{1}),\dots,\iota(\varsigma_{L})\}$ of tangent
to $\mathbb{S}^{2}$ lines. The manifold of all such six-tuples will be denoted by $M^{6}.$

\vskip .1cm
Let $\varsigma^{\prime},\varsigma^{\prime\prime}$ be two equal cylinders of radius $r$ touching $\mathbb{S}^{2},$ which also touch each other, while
$\iota^{\prime},\iota^{\prime\prime}$ are the corresponding tangents to $\mathbb{S}^{2}.$ If $d=d_{\iota^{\prime}\iota^{\prime\prime}}$ is the
distance between $\iota^{\prime},\iota^{\prime\prime}$ then we have 
\[ r=\frac{d}{2-d}\ ,\]
The study of the manifold of six-tuples of cylinders of equal radii, some of which are touching, is equivalent to the study the manifold $M^{6}$ and
the function $D$ on it, defined by 
\[ D\left(  \iota_{1},...,\iota_{6}\right)  =\min_{1\leq i<j\leq6}d_{\iota
_{i}\iota_{j}}\ .\]

\subsection{Configuration manifold}
\label{sectConfManifold}
Here we collect the notation of \cite{OS}.

\vskip.2cm 
Let $\mathbb{S}^{2}\subset\mathbb{R}^{3}$ be the unit sphere, centered at the origin. For every $x\in\mathbb{S}^{2}$ we denote by $TL_{x}$
the set of all (unoriented) tangent lines to $\mathbb{S}^{2}$ at $x.$ We denote by $M$ the manifold of tangent lines to $\mathbb{S}^{2}.$ We represent
a point in $M$ by a pair $\left(  x,\xi\right)  $, where $\xi$ is a unit tangent vector to $\mathbb{S}^{2}$ at $x,$ though such a pair is not unique:
the pair $\left(  x,-\xi\right)  $ is the same point in $M.$ 

\vskip .1cm
We shall use the following coordinates on $M$. Let $\mathbf{x,y,z}$ be the standard coordinate
axes in $\mathbb{R}^{3}$. Let $R_{\mathbf{x}}^{\alpha}$, $R_{\mathbf{y} }^{\alpha}$ and $R_{\mathbf{z}}^{\alpha}$ be the counterclockwise rotations
about these axes by an angle $\alpha$, viewed from the tips of axes. We call the point $\mathsf{N}=\left(  0,0,1\right)  $ the North pole, and
$\mathsf{S}=\left(  0,0,-1\right)  $ -- the South pole. By \textit{meridians} we mean geodesics on $\mathbb{S}^{2}$ joining the North pole to the South
pole. The meridian in the plane $\mathbf{xz}$ with positive $\mathbf{x}$ coordinates will be called Greenwich. The angle $\varphi$ will denote the
latitude on $\mathbb{S}^{2},$ $\varphi\in\left[  -\frac{\pi}{2},\frac{\pi} {2}\right]  ,$ and the angle $\varkappa\in\lbrack0,2\pi)$ -- the longitude, so
that Greenwich corresponds to $\varkappa=0.$ Every point $x\in\mathbb{S}^{2}$ can be written as $x=\left(  \varphi_{x},\varkappa_{x}\right)$.

\vskip .1cm
Finally, for each $x\in\mathbb{S}^{2}$, we denote by $R_{x}^{\alpha}$ the rotation by the angle $\alpha$ about the axis joining $\left(  0,0,0\right)  $ to $x,$
counterclockwise if viewed from its tip, and by $\left(  x,\uparrow\right) $ we denote the pair $\left(  x,\xi_{x}\right)  ,$ $x\neq\mathsf{N,S,}$ where
the vector $\xi_{x}$ points to the North. We also abbreviate the notation $\left(  x,R_{x}^{\alpha}\uparrow\right)  $ to $\left(  x,\uparrow_{\alpha
}\right)  $.

\vskip.2cm 
Let $u=\left(  x^{\prime},\xi^{\prime}\right)  ,$ $v=\left( x^{\prime\prime},\xi^{\prime\prime}\right)  $ be two lines in $M$. We denote
by $d_{uv}$ the distance between $u$ and $v$; clearly $d_{uv}=0$ iff $u\cap v\neq\varnothing.$ If the lines $u,v$ are not parallel then the square of
$d_{uv}$ is given by the formula
\begin{equation}\label{formdist}
d_{uv}^{2}=\frac{\det^{2}[\xi^{\prime},\xi^{\prime\prime},x^{\prime\prime
}-x^{\prime}]}{1-(\xi^{\prime},\xi^{\prime\prime})^{2}}\ ,\end{equation}
where $(\ast,\ast)$ is the scalar product. 

\vskip.2cm 
We are studying the critical points of the function
\[ D\left(  \mathbf{m}\right)  =\min_{1\leq i<j\leq N}d_{u_{i}u_{j}}\ ,\]
on the manifold $M^{N}$ is of N-tuples
\begin{equation} \mathbf{m}=\left\{  u_{1},...,u_{N}\,:\, u_{i}\in M\, ,\, i=1,...,N\right\}  .\end{equation}
 
The norm which we mentioned in the Introduction is defined by
\begin{equation} \left\Vert u-v\right\Vert =\left\Vert x^{\prime}-x^{\prime\prime}\right\Vert
+\min\left\{  \left\Vert \xi^{\prime}-\xi^{\prime\prime}\right\Vert
,\left\Vert \xi^{\prime}+\xi^{\prime\prime}\right\Vert \right\}  .\label{norm} \end{equation}
 
\subsection{Configuration $O_{6}$}\label{configO6}
We denote by $e_i$ the orthonormal basis in $\mathbb{R}^3$,
$$e_1\equiv e_x=(1,0,0)\ ,\ e_2\equiv e_y=(0,1,0)\ ,\ e_3\equiv e_z=(0,0,1)\ .$$
Let $\varrho$ be the rotation of order 3, which cyclically permutes the vectors $e_i$,
$$\varrho\colon e_1\mapsto e_2\mapsto e_3\mapsto e_1\ ,$$
$\mathcal{I}$ the central reflection,
$$\mathcal{I}v=-v\ ,\ v\in \mathbb{R}^3\ ,$$
and $r$ the rotation around the axe $Ox$ by the angle $\pi$,
$$r\colon e_1\mapsto e_1\ ,\ e_2\mapsto -e_2\ ,\ e_3\mapsto -e_3\ .$$
The maps $\varrho$, $\mathcal{I}$ and $r$ generate the group $\mathbb{A}_4\times C_2$ where $\mathbb{A}_4$ is the alternating group on four letters, generated by $\varrho$ and $r$,
and $C_2$ is the cyclic group of order 2, generated by $\mathcal{I}$.

\vskip .1cm
Let $\ell_{1}^+$ be the line in the direction $e_3$ touching the unit sphere at the point $e_1$ and let $\ell_{2}^+=\varrho\ell_{1}^+$, $\ell_{3}^+=\varrho^2\ell_{1}^+$. The images of the lines $\ell_j^+$, $j=1,2,3$, under the central reflection $\mathcal{I}$ will be denoted by $\ell_j^-$,
 $$\ell_j^-=\mathcal{I}\ell_j^+\ ,\ j=1,2,3\ .$$
The six lines $\ell_j^+$, $\ell_j^-$, $j=1,2,3$, form the configuration $O_6$. The symmetry group of the configuration $O_6$ is $\mathbb{A}_4\times C_2$.

\vskip .1cm
Let $O_6(t)$ be a deformation of the configuration $O_6$. We have fifteen pairwise distances between the lines in $O_6(t)$. There are three
distances $d_{\ell_j^+(t),\ell_j^-(t)}$, $j=1,2,3$, which do not have a well defined limit when $t\to 0$ because the lines $\ell_j^+$ and $\ell_j^-$
are parallel. The remaining twelve distances do have a well defined limit, equal to 1, when $t\to 0$. We shall first study these twelve distances.  
More generally, let $\mathbf{m}$ be a point in a small enough neighborhood of $O_{6}$ and let $\widetilde{\ell}_j^+$, $\widetilde{\ell}_j^-$, $j=1,2,3$ be the positions of perturbed lines 
$\ell_j^+$, $\ell_j^-$, $j=1,2,3$. Let 
\begin{equation}\label{oprtid}\widetilde{D}(\mathbf{m}):=\min_{1\leq i<j\leq 3,\epsilon,\epsilon'=\pm} \left(d_{\widetilde{\ell}_i^\epsilon ,\widetilde{\ell}_j^{\epsilon' }} \right)\ .\end{equation}

\begin{theorem}\label{localmaxO6}
The configuration $O_{6}\ $is a point of a sharp local maximum of the function $\widetilde{D}$:
for any point $\mathbf{m}$ in a vicinity of $O_{6}$ we have
\[ \widetilde{D}\left(  \mathbf{m}\right)  <1=\widetilde{D}\left(  O_{6}\right) \ .\]
\end{theorem}

We have $\widetilde{D}(O_6)=D(O_6)=1$ and $D(\mathbf{m})\leq \widetilde{D}(\mathbf{m})$.
This implies that the configuration $O_6$ is locally maximal.

\begin{corollary} The configuration $O_{6}\ $is a point of a sharp local maximum of the function $D$. \end{corollary}

In the process of the proof of Theorem \ref{localmaxO6}, we will see that there exists a 6-dimensional subspace $L_{quadr}$ in the
tangent space to $M^{6}\,\operatorname{mod}SO\left(  3\right)  $ at $O_{6}$, such that for any $l\in L_{quadr}$, $\left\Vert l\right\Vert =1$, we have
\[ -c_u \left\Vert l \right\Vert t^{2}\leq \widetilde{D}\left(  O_{6}+tl\right)
-\widetilde{D}\left(  O_{6}\right)  \leq -c_d \left\Vert l \right\Vert t^{2} \]
for $t$ small enough. Here $c_d$ and $c_u$ are some constants, $0<c_d \leq c_u <+\infty$ and $O_{6}+tl\in M^{6}\,\operatorname{mod}SO\left(  3\right)  $
stands for the exponential map applied to the tangent vector $tl$. 

\vskip .1cm
For the tangent vectors outside $L_{quadr}$ we have
\[ -c_u^{\prime}\left(  l\right)  t\leq \widetilde{D}\left(  O_{6}+tl\right)  -\widetilde{D}\left(
O_{6}\right)  \leq -c_d^{\prime}\left(  l\right)  t,\]
where now $c_d^{\prime}\left(  l\right)$ and $c_u^{\prime}\left(  l\right)$ are some positive valued functions of $l$, $0<c^{\prime}\left(  l\right)  \leq c^{\prime\prime}\left(  l\right)<+\infty$.

\subsection{Connected components}
In \cite{OS-C6} we have shown that the configuration $C_\mathfrak{m}$ is a local maximum. Together with Theorem \ref{localmaxO6} we arrive at the following conclusion.

\begin{corollary}\label{cornonconne}
The configuration space $\mathcal{Z}$ of six unit cylinders touching a unit ball is not connected.\end{corollary}

\begin{proof} Theorem \ref{localmaxO6} implies that the configuration $O_6$ is an isolated point in the space $\mathcal{Z}$ mod $SO(3)$. Hence the configurations
$$\gamma(\varphi)=C_{6}\bigl(\varphi,\delta\left(  \varphi\right)
,\varkappa\left(  \varphi\right)  \bigr)\ ,\ \varphi\in\left[  0;\arcsin \left(\frac{\sqrt{5}}{4}\right)\right]\ ,$$
constructed in \cite{OS}, belong to another component of $\mathcal{Z}$ (at $\phi=\arcsin \left(\frac{\sqrt{5}}{4}\right)$ the function $D\bigl(  \gamma(\varphi)\bigr)$ gets back its
initial value 1, see the end of Section 5 in \cite{OS}).
\end{proof}

\vskip .3cm
\noindent{\bf Remark.} In contrast to the configuration $O_6$, the configuration $C_\mathfrak{m}$ is not congruent to its mirror image. To show this, 
we need several definitions. 

\vskip .1cm
A triple $\mathcal{T}$ of straight lines is said to be in a generic position if there is no plane parallel to all three lines. A triple $\mathcal{T}$ in a generic position defines an orientation of the space
$\mathbb{R}^3$, or, if an orientation of $\mathbb{R}^3$ is given, a sign $\sigma (\mathcal{T})$. This sign is defined as follows. There is a unique hyperboloid $\mathcal{H}(\mathcal{T})$ of one sheet, passing through 
the three straight lines of $\mathcal{T}$. Let the equation of the hyperboloid $\mathcal{H}(\mathcal{T})$, in its principal axes, be $z^2=ax^2+by^2-r^2$. The hyperboloid $\mathcal{H}(\mathcal{T})$ has two families of rulings and the three lines of $\mathcal{T}$ belong to 
the same family. So, viewed from a remote point on the $z$-axis (that is, a point $(0,0,\mathsf{z})$ with $\mathsf{z}$ big enough),  
we will see a picture of the three lines of $\mathcal{T}$ isotopic to the one shown on Figure \ref{SkewTriples}.
The isotopy class of the picture will not change if we are looking from the minus $\infty$ of the $z$-axis, that is, from a point $(0,0,-\mathsf{z})$ with $\mathsf{z}$ big enough. 
The sign $\sigma (\mathcal{T})$ associated to the triple $\mathcal{T}$ is $+$ if we see the left picture; for the right picture the sign is $-$.
\begin{figure}[H]
\vspace{.1cm} \centering
\includegraphics[scale=0.3]{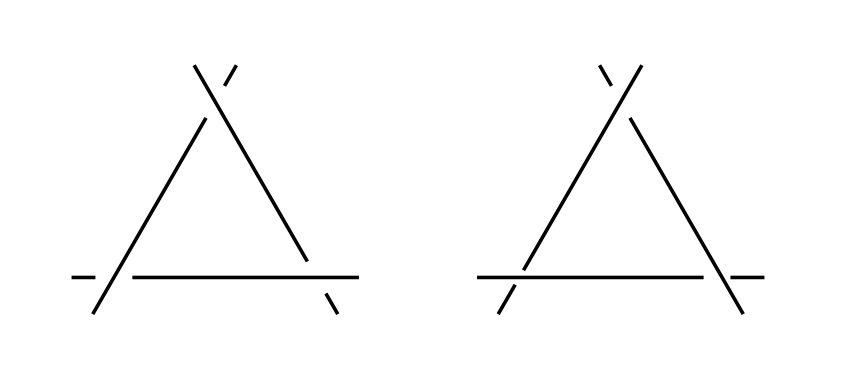}
\caption{Three skew lines}
\label{SkewTriples}
\end{figure}

\vskip .1cm
There is a way, see e. g. \cite{VV}, to calculate the sign $\sigma (\mathcal{T})$ without determining the hyperboloid $\mathcal{H}(\mathcal{T})$. This is done as follows. First one associates a sign to a pair of oriented skew straight lines, as shown on Figure 
\ref{Skewlines+-}. 
\begin{figure}[H]
\vspace{.1cm} \centering
\includegraphics[scale=0.3]{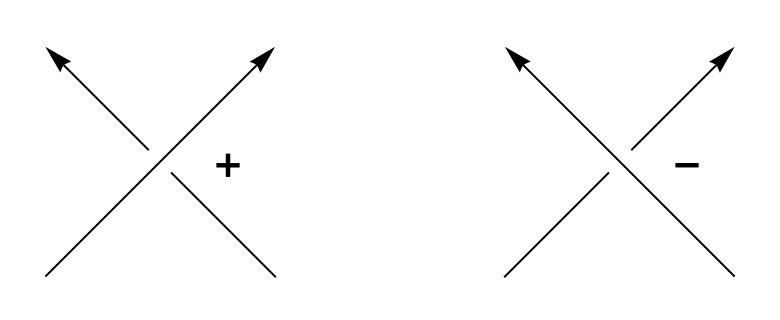}
\caption{Oriented lines}
\label{Skewlines+-}
\end{figure}

\vskip .1cm
Given a triple $\mathcal{T}$, equip the lines of $\mathcal{T}$ arbitrarily with an orientation. Then the sign $\sigma (\mathcal{T})$ is equal to the product of the signs for the three pairs of 
these oriented lines. 

\vskip .1cm
The configuration $C_\mathfrak{m}$ has a three-fold axis of symmetry. The action of the cyclic group $\mathcal{C}_3$ on the set of six cylinders of $C_\mathfrak{m}$ 
has two orbits of length three. These are the brown and the red triplets of cylinders on Figure \ref{confCm}. The sign of both triples is negative. For the configuration, obtained by a central symmetry 
from the configuration $C_\mathfrak{m}$, the sign of the both reflected triples is positive. Therefore, the configuration $C_\mathfrak{m}$ is not congruent to its mirror image.
Let us for the moment call $C_\mathfrak{m}^+$ the configuration on Figure \ref{confCm}, and $C_\mathfrak{m}^-$ its mirror image.

\vskip .1cm
Let $\mathcal{Z}(\mathsf{\geq R})$, respectively $\mathcal{Z}(\mathsf{> R})$, denote the configuration space, mod $SO(3)$, of $6$ cylinders,
of radius bigger or equal to $\mathsf{R}$, respectively, bigger than $\mathsf{R}$, touching the unit ball. 

\vskip .1cm
In the space $\mathcal{Z}(\mathsf{\geq 1})$, the configurations  $C_\mathfrak{m}^+$ and  $C_\mathfrak{m}^-$ are in the same connected component: one can move from
 $C_\mathfrak{m}^+$ to $C_6$ and then return to  $C_\mathfrak{m}^-$. 
 
\vskip .1cm 
 \noindent{\bf Conjecture.} The configurations $C_\mathfrak{m}^+$ and  $C_\mathfrak{m}^-$  belong to different connected components of the space 
$\mathcal{Z}(\mathsf{>1})$. 

\vskip .1cm 
The following observation might be helpful in testing this conjecture. There are twenty different triples of tangent lines in the configuration $C_\mathfrak{m}^+$. Among them there are twelve 
positive triples and eight negative triples. Therefore, in a motion from $C_\mathfrak{m}^+$ to  $C_\mathfrak{m}^-$ some triples have to pass a non-generic position; the formulas for the distances 
between the cylinders slightly simplify when there is a non-generic triple. 

\vskip .1cm
Also, one can ask the following natural questions. Let $\mathcal{D}$ be a configuration of non-overlapping cylinders of the same radius. Supply each cylinder with an orientation. 
Let $\mathsf{p}$ be a path in the space $\mathcal{Z}(\mathsf{\geq R})$ or $\mathcal{Z}(\mathsf{> R})$
which starts and ends with the configuration $\mathcal{D}$. In general, the path $\mathsf{p}$ might permute the cylinders or change their orientation. 
The first question is -- what is the group of permutations and orientation changes induced by all possible such paths? 

\vskip .1cm
We conjecture that for the configuration $C_\mathfrak{m}$ the only permutations of oriented cylinders 
which can be achieved by a motion in the space $\mathcal{Z}(\mathsf{\geq 1})$ are the rigid rotations from the dihedral group $\mathbb{D}_3$.

\vskip .1cm
The group generated by permutations and orientation changes of  six cylinders is the wreath product $\mathbb{S}_6\wr\, \mathcal{C}_2$. It is interesting to know what is the maxi\-mal radius $\mathsf{R}$ of cylinders for which the whole group $\mathbb{S}_6\wr \mathcal{C}_2$ is realizable by paths in $\mathcal{Z}(\mathsf{\geq R})$. 

\vskip .1cm
Let us say that a subgroup $\mathcal{H}$ of $\mathbb{S}_6\wr\, \mathcal{C}_2$ is path-realizable if there exists a configuration $\mathcal{D}$ of six 
non-overlapping cylinders in $\mathcal{Z}(\mathsf{\geq R})$ for some $\mathsf{R}$, such that the elements of 
$\mathbb{S}_6\wr\, \mathcal{C}_2$, realizable by paths in $\mathcal{Z}(\mathsf{\geq R})$, form the subgroup $\mathcal{H}$. Which subgroups
 of $\mathbb{S}_6\wr\, \mathcal{C}_2$ are path-realizable? For example, for twelve spheres of radius slightly larger than one, touching the unit sphere,  it is known that the subgroup $\mathbb{A}_{12}$ of  $\mathbb{S}_{12}$ is path-realizable, see Appendix to Chapter 1 in \cite{CS}.
What is the maximal radius $\mathsf{R}$ for each path-realizable subgroup of $\mathbb{S}_6\wr\, \mathcal{C}_2$? 
Does this maximal radius $\mathsf{R}$ depend on the connected component, to which the configuration $\mathcal{D}$ belongs, of the space $\mathcal{Z}(\mathsf{\geq R})$? 
Of course, these questions can be asked about any number of cylinders, not necessarily six.

\section{Criticality of $O_{6}$}\label{maxconfO6}
We shall study the deformed configuration $O_6(t)$, formed by the tangent lines
$$\ell_{j}^\epsilon(t)=R_{\varrho^2 e_j}^{a_j^\epsilon\, t}\, R_{\varrho e_j\phantom{^2}}^{b_j^\epsilon\, t}\, R_{e_j\phantom{^2}}^{c_j^\epsilon\, t}\, \ell_{j}^\epsilon\ ,\ j=1,2,3\ ,\ \epsilon\in \{+,-\}\ ,$$
where we write $R_{e_1}^\alpha$ (respectively, $R_{e_2}^\alpha$ and $R_{e_3}^\alpha$)
for $R_{\mathbf{x}}^{\alpha}$ (respectively, $R_{\mathbf{y}}^{\alpha}$ and $R_{\mathbf{z}}^{\alpha}$).

\vskip .1cm
To fix the rotational symmetry we keep the tangent line $\ell_1^+$ at its place, that is, $a_1^+=0$, $b_1^+=0$ and $c_1^+=0$.

\vskip .1cm
We shall be studying the variations of twelve distances appearing in the function $\widetilde{D}$, defined by the formula (\ref{oprtid}).

\subsection{First differentials}\label{secdif}
First we calculate (directly) the deformations of the squares of distances in the first order.
For brevity, we shall write $[d^2_{u,v}]_j$ for the coefficient at $t^j$ in the function $d^2_{u(t),v(t)}$ where $u,v\in O_6(t)$. Here is the result:
$$\begin{array}{c}
 [d^2_{\ell_1^+,\ell_2^-}]_1=-2b_{2 }^-\ ,\ [d^2_{\ell_1^+,\ell_2^+}]_1=2b_{2 }^+\ ,\\[1em]
 \lbrack d^2_{\ell_1^+,\ell_3^-}\rbrack_1=2a_{3 }^-\ ,\ [d^2_{\ell_1^+,\ell_3^+}]_1=-2a_{3 }^+\ ,\end{array}$$
\begin{equation}\label{disjvar}
\begin{array}{c}
 \lbrack d^2_{\ell_1^-,\ell_2^-}\rbrack_1=2(b_{2 }^- -a_{1 }^-)\ ,\ [d^2_{\ell_1^-,\ell_2^+}]_1=2(a_{1 }^- -b_{2 }^+)\ ,\\[1em]
 \lbrack d^2_{\ell_1^-,\ell_3^-} \rbrack_1=2(b_{1 }^- -a_{3 }^-)\ ,\ [d^2_{\ell_1^-,\ell_3^+}]_1=2(a_{3 }^+ -b_{1 }^-)\ ,
\end{array}\end{equation}
$$
\begin{array}{c}
 \lbrack d^2_{\ell_3^-,\ell_2^-}\rbrack_1=2(b_{3 }^- -a_{2 }^-)\ ,\ [d^2_{\ell_3^-,\ell_2^+}]_1=2(a_{2 }^+ -b_{3 }^-)\ ,\\[1em]
 \lbrack d^2_{\ell_3^+,\ell_2^-}\rbrack_1=2(a_{2 }^- -b_{3 }^+)\ ,\ [d^2_{\ell_3^+,\ell_2^+}]_1=2(b_{3 }^+ -a_{2 }^+)\ .
 \end{array}$$
The expressions in the first two lines are shorter because the the tangent line $\ell_1^+$ does not move.

\vskip .1cm
The above differentials are not independent. All the linear relations
between them are linear combinations of the following three relations:
\begin{eqnarray}\label{depe1}&[d^2_{\ell_1^+,\ell_2^-}]_1+[d^2_{\ell_1^+,\ell_2^+}]_1+[d^2_{\ell_1^-,\ell_2^-}]_1+[d^2_{\ell_1^-,\ell_2^+}]_1
=0\ ,&\\[1em]
\label{depe2}&[d^2_{\ell_1^+,\ell_3^-}]_1+[d^2_{\ell_1^+,\ell_3^+}]_1+[d^2_{\ell_1^-,\ell_3^-}]_1+[d^2_{\ell_1^-,\ell_3^+}]_1
=0\ ,\\[1em]
\label{depe3}
&\lbrack d_{\ell_{2}^{-},\ell_{3}^{-}}^{2}]_{1}+[d_{\ell_{2}^{-},\ell_{3}^{+}
}^{2}]_{1}+[d_{\ell_{2}^{+},\ell_{3}^{-}}^{2}]_{1}+[d_{\ell_{2}^{+},\ell
_{3}^{+}}^{2}]_{1}=0\ .&
\end{eqnarray}

It follows that the distances will not decrease in the first order only if
$$\begin{array}{c}0\geq b_{2 }^- \geq a_{1 }^- \geq b_{2 } ^+\geq 0\ ,\\[.8em]
0\geq a_{3 } ^+ \geq b_{1 } ^- \geq a_{3 }^-\geq 0\ ,\\[1.em]
 b_{3 }^- \geq a_{2 }^- \geq b_{3 }^+\geq a_{2 }^+\geq b_{3 }^-\ .\end{array}$$
Therefore we must have
\begin{eqnarray}\label{rela1}&b_{2 }^- = a_{1 }^- = b_{2 }^+= 0\ ,&\\[.6em]
\label{rela1b}&a_{3 }^+ = b_{1 }^- = a_{3 }^-= 0\ ,&\\[.6em]
\label{rela2} &b_{3 }^- = a_{2 }^- = b_{3 }^+= a_{2 }^+=\omega\ ,&\end{eqnarray}
where we have denoted by $\omega$ the common value of the four
equal coefficients. In the remaining regime, defined by nine relations (\ref{rela1})--(\ref{rela2}), the linear contribution to the 12 distances vanish.

\paragraph{Remark.} As mentioned in the Introduction, the important object in the questions concerning local maxima of a minimum of several analytic functions is the vector subspace $E$, on which all differential vanish, of the tangent space to a given configuration. In our situation the subspace $E$ is 
given by the equations (\ref{rela1})--(\ref{rela2}). 

\vskip .1cm
One may wonder if a consideration of the remaining three distances 
$d_{\ell_{1}^{+},\ell_{1}^{-}}$, $d_{\ell_{2}^{+},\ell_{2}^{-}}$ and $d_{\ell_{3}^{+},\ell_{3}^{-}}$ could simplify the analysis (being 2 at $t=0$, these distances can drop to 0 under an infinitesimal deformation). The answer is negative: it turns out 
that these three distances keep their value 2 under the infinitesimal variations from the subspace $E$. For example, the positions of the cylinders 
$\ell_{2}^+(t)$ and $\ell_{2}^-(t)$ along a deformation from $E$ are 
$$\ell_{2}^+(t)=R_{e_1}^{\omega t}\, R_{e_2}^{c_2^+ (t)}\, \ell_{2}^+\ ,\ \ell_{2}^-(t)=R_{e_1}^{\omega t}\, R_{e_2}^{c_2^- (t)}\, \ell_{2}^-\ .$$
The distance between the tangent lines $\ell_{2}^+(t)$ and $\ell_{2}^-(t)$ is the same as the distance between the tangent lines $\check{\ell}_{2}^+(t):=R_{e_2}^{c_2^+ (t)}\, \ell_{2}^+$ and 
$\check{\ell}_{2}^-(t):=R_{e_2}^{c_2^- (t)}\, \ell_{2}^-$. But the tangent lines $\check{\ell}_{2}^+(t)$ and $\check{\ell}_{2}^-(t)$ stay parallel to the plane $Oxz$ and the distance between them remains to be 2. 

\vskip .1cm
One may be tempted to think that each of thee groups of equalities (\ref{rela1})--(\ref{rela2}) is responsible for leaving fixed exactly one of three distances
$d_{\ell_{1}^{+},\ell_{1}^{-}}$, $d_{\ell_{2}^{+},\ell_{2}^{-}}$ and $d_{\ell_{3}^{+},\ell_{3}^{-}}$. This is however not the case.

\subsection{Second differentials}\label{secdifb}
We now consider the same combinations (\ref{depe1})--(\ref{depe3}) but for the coefficients in $t^2$.
Clearly, these combinations will contain only six parameters: $\omega$ and all
$c_{j1}^{\epsilon}$, $j=1,2,3,$ $\epsilon\in\{+,-\}$, except $c_{1 }^{+},$ which is fixed
to be zero. Explicitly (this is again a direct calculation) these combinations read
\begin{equation}
\begin{array}{rrl}
\Upsilon_1&:=&\frac{1}{2}\left( [d^2_{\ell_1^+,\ell_2^-}]_2+[d^2_{\ell_1^+,\ell_2^+}]_2+[d^2_{\ell_1^-,\ell_2^-}]_2+[d^2_{\ell_1^-,\ell_2^+}]_2\right)\\[1em]
&=&c_{1 }^- c_{2 }^+ -(c_{1 }^-)^2-c_{1 }^- c_{2 }^-+2c_{1 }^- \omega-2\omega^2\ .\end{array}
\label{55}\end{equation}
\begin{equation}
\begin{array}{rrl}
\Upsilon_2&:=&\frac{1}{2}\left( [d^2_{\ell_1^+,\ell_3^-}]_2+[d^2_{\ell_1^+,\ell_3^+}]_2+[d^2_{\ell_1^-,\ell_3^-}]_2+[d^2_{\ell_1^-,\ell_3^+}]_2\right)\\[1em]
&=&c_{1 }^- c_{3 }^+ - (c_{3 }^-)^2 - c_{1 }^- c_{3 }^- -(c_{3 }^+)^2\ ,\end{array}
\label{56}\end{equation}
\begin{equation}
\begin{array}{rrl}
\Upsilon_3&:=&\frac{1}{2}\left( [d^2_{\ell_3^-,\ell_2^-}]_2+ [d^2_{\ell_3^-,\ell_2^+}]_2+[d^2_{\ell_3^+,\ell_2^-}]_2+ [d^2_{\ell_3^+,\ell_2^+}]_2\right)\\[1em]
&=&c_{2 }^- c_{3 }^+ +c_{2 }^+ c_{3 }^--c_{2 }^- c_{3 }^- -c_{2 }^+ c_{3 }^+ -(c_{2 }^-)^2-(c_{2 }^+)^2\ .\end{array}
\label{57}\end{equation}

The distances will not decrease in the second order only if
\begin{equation}\Upsilon_1\geq 0\ ,\ \Upsilon_2\geq 0\ ,\ \Upsilon_3\geq 0\ .\label{22}\end{equation}
We will show now that the system $\left(  \ref{22}\right)  $ has only zero
solution. We rewrite it in the form
\begin{equation}\label{combsecor1}\omega^2+(\omega-c_{1 }^-)^2\leq c_{1 }^- (c_{2 }^+ - c_{2 }^-)\ ,\end{equation}
\begin{equation}\label{combsecor2}(c_{3 }^-)^2+(c_{3 }^+)^2\leq c_{1 }^- (c_{3 }^+ - c_{3 }^-)\ ,\end{equation}
\begin{equation}\label{combsecor3}(c_{2 }^-)^2+(c_{2 }^+)^2\leq (c_{2 }^- -c_{2 }^+) (c_{3 }^+ - c_{3 }^-)\ .\end{equation}
The left hand sides are non-negative. Taking the product of the inequalities (\ref{combsecor1}) and 
(\ref{combsecor2}), we find
$$(c_{1 }^-)^2 (c_{2 }^+ - c_{2 }^-)(c_{3 }^+ - c_{3 }^-)\geq 0\ .$$
Assume that $c_{1 }^-\neq 0$.
Then $(c_{2 }^+ - c_{2 }^-)(c_{3 }^+ - c_{3 }^-)\geq 0$. But
(\ref{combsecor3}) implies that $(c_{2 }^+ - c_{2 }^-)(c_{3 }^+ - c_{3 }^-)\leq 0$.
Therefore $(c_{2 }^+ - c_{2 }^-)(c_{3 }^+ - c_{3 }^-)= 0$. Now
it follows from (\ref{combsecor3}) that $c_{2 }^-=c_{2 }^+=0$. Then (\ref{combsecor1})
implies that $c_{1 }^- = 0$.

Thus, we have checked that $c_{1 }^- $ must be 0. Now (\ref{combsecor1}) implies that
\begin{equation}\label{doprel1}\omega=0\ ,\end{equation}
 (\ref{combsecor2}) implies that
 \begin{equation}\label{doprel2}c_{3 }^-=c_{3 }^+=0\end{equation}
 and then (\ref{combsecor3}) implies that
 \begin{equation}\label{doprel3}c_{2 }^-=c_{2 }^+=0\ .\end{equation}
The above computations show that along any path with tangent vector in the 6-dimensional subspace (\ref{rela1})-(\ref{rela2}) our
function decays as $t^{2}$.

\vskip .1cm
Together, equalities (\ref{rela1})--(\ref{rela2}) and (\ref{doprel1})--(\ref{doprel3}) show that order $t^1$ coefficients of all functions $a_j^\epsilon (t)$,
$b_j^\epsilon (t)$ and $c_j^\epsilon (t)$ vanish.

\vskip .1cm
This is not, however, the end of the story, since we do not have uniform estimates on the lengths of all the paths entering into
our argument. In general, it is possible that a $\mathcal{C}^{\infty}$ function decays along any analytic path starting at the origin, yet it
increases along a non-analytic path as the following example shows.

\vskip .1cm
\noindent{\bf Example.} Let us draw two graphs on the plane $\mathbb{R}^2$, of functions $f_1(x)=e^{-1/x}$ and $f_2(x)=e^{-2/x}$ for $x\geq 0$. An example is provided by an arbitrary $\mathcal{C}^{\infty}$ function in a vicinity of origin in $\mathbb{R}^2$ which increases in the horn between the graphs of $f_1$ and $f_2$ and decreases otherwise: for any analytic path, starting at the origin, there exists a duration when the path does not enter the interior of the horn. 

\vskip .1cm
So to complete the argument we use the theorem \ref{lq2b}, Section \ref{genforms}.

\subsection{Three forms}\label{threeforms} 
A straightforward calculation shows that each of the forms $\Upsilon_{1},\Upsilon_{2},\Upsilon_{3}$ has the matrix rank three.

\vskip .1cm
If there existed a negatively defined strictly convex combination of three forms
$\Upsilon_{1},\Upsilon_{2},\Upsilon_{3}$ then we could directly refer to Theorem 2, Section 5 of \cite{OS-C6} to finish the proof of Theorem \ref{localmaxO6}. 
Besides, an existence of such combination would give an easier proof of the statement that the system $\Upsilon_{1}\geq 0$, $\Upsilon_{2}\geq 0$ and $\Upsilon_{3}\geq 0$ admits only a trivial solution. However such combination does not exist as we will now show.

\begin{proposition} There is no positively defined convex linear combination of the three forms $\Upsilon_{1},\Upsilon_{2}$ and $\Upsilon_{3}$. 
\end{proposition}

\begin{proof}
Let $\widetilde{\omega}=\omega-\frac{1}{2}c_{1 }^-$. We have
$$\Upsilon_1=\widetilde{\Upsilon}_1-2\widetilde{\omega}^2\ ,$$
where
$$\widetilde{\Upsilon}_1=c_{1 }^- c_{2 }^+ -\frac{1}{2}(c_{1 }^-)^2-c_{1 }^- c_{2 }^-\ .$$
The variable $\widetilde{\omega}$ is not involved in the forms $\Upsilon_2$ and $\Upsilon_3$. Therefore, a convex combination of the forms $(-\Upsilon_1)$, $(-\Upsilon_2)$ and 
$(-\Upsilon_3)$ is positively defined on the subspace with coordinates $\{\widetilde{\omega}, c_{1 }^- , c_{2 }^+ ,c_{2 }^- ,c_{3 }^+ ,c_{3 }^-\}$ if and only if the same 
convex combination of the forms $(-\widetilde{\Upsilon}_1)$, $(-\Upsilon_2)$ and 
$(-\Upsilon_3)$ is positively defined on the five-dimensional subspace with coordinates $\{ c_{1 }^- , c_{2 }^+ ,c_{2 }^- ,c_{3 }^+ ,c_{3 }^-\}$. Thus it is sufficient to consider only this 
five-dimensional space. 

\vskip .1cm
Assume that a combination 
$$\Upsilon:= -\bigl( \widetilde{\Upsilon}_{1}+\alpha\Upsilon_{2}+\beta\Upsilon_{3}\bigr)\ ,$$
where $\alpha,\beta>0$, is positively defined; without loss of generality we fixed the coefficient of the form $(-\widetilde{\Upsilon}_{1})$ to be 1.

\vskip .1cm 
In the coordinates $\{ c_{1 }^- , c_{2 }^+ ,c_{2 }^- ,c_{3 }^+ ,c_{3 }^-\}$ the form $\Upsilon$ has the following Gram matrix:

$$\frac{1}{2}\left(\begin{array}{rrrrr}
1&-1&1&-\alpha&\alpha\\
-1&2\beta&0&\beta&-\beta\\
1&0&2\beta&-\beta&\beta\\
-\alpha&\beta&-\beta&2\alpha&0\\
\alpha&-\beta&\beta&0&2\alpha
\end{array}\right)\ .$$

\vskip .3cm
The Sylvester criterion says that the positivity of the form $\Upsilon$ is equivalent to the following system of inequalities 
\begin{equation}\label{sylvsyst}\begin{array}{c}
2\beta-1>0\ ,\ 4\beta(\beta-1)>0\ ,\ -4\beta(2\alpha-4\alpha\beta+\alpha^2\beta+\beta^2)>0\ ,\\[1em]
 -16\alpha\beta(\alpha-3\alpha\beta+\alpha^2\beta+\beta^2)>0\ .\end{array}\end{equation}
Taking into account that the coefficients $\alpha$ and $\beta$ are positive, the system (\ref{sylvsyst}) reduces to the system  
$$\beta>1\ ,\ 2\alpha-4\alpha\beta+\alpha^2\beta+\beta^2<0\ ,\ \alpha-3\alpha\beta+\alpha^2\beta+\beta^2<0\ ,$$
which is incompatible. Already the first and the third inequalities are not compatible. Indeed, consider the left hand side 
$$\mathsf{m}:=\alpha-3\alpha\beta+\alpha^2\beta+\beta^2$$ of the third inequality as the polynomial in $\alpha$. The quadratic polynomial $\mathsf{m}$ can take a negative 
value only if its roots are real. However the discriminant of the polynomial  $\mathsf{m}$ is 
$$-(\beta-1)^2(4\beta-1)\ ,$$
which is negative for $\beta>1$.
\end{proof}

\section{Sufficient condition}\label{genforms}
Let $\{F_{1}\left(  x\right)  ,\dots,F_{m}\left(  x\right)\}$ be a family of functions $\mathsf{U}\to \mathbb{R}$, where 
$\mathsf{U}\subset\mathbb{R}^n$  is a neighborhood of the origin $0\in\mathbb{R}^n$, such that $F_u(0)=0$, $u=1,\dots ,m$.
We assume that the number of functions is not greater that the number of variables, $m\leq n$.

\vskip .1cm
For the configurations of tangent lines in Theorem \ref{localmaxO6} 
the functions $F_u$ are the differences between the squares of distances in the perturbed and non-perturbed configurations.

\vskip .1cm
We are studying the function
\[ {\sf F}\left(  x\right)  :=\min\left\{  F_{1}\left(  x\right)
,\dots,F_{m}\left(x\right)  \right\} \ .\]

In \cite{OS-C6} we have proved the local maximality of the configuration $C_{\mathfrak{m}}$. For the configuration $C_{\mathfrak{m}}$
there is exactly one convex linear dependency between the differentials of the functions $F_u(x)$, $u=1,\dots ,m$, at the origin. We have given in \cite{OS-C6} a sufficient condition ensuring that the point $\mathbf{0}\in\mathbb{R}^{n}$ is a sharp local maximum
of the function ${\sf F}\left(  x\right)$.

\vskip .1cm
As we have seen in Section \ref{maxconfO6}, the space of linear dependencies between $dF_u(0)$, $u=1,\dots ,m$, is three-dimensional and 
has a basis consisting of three convex dependencies. 
Moreover in our coordinates the groups of coordinates entering these linear combinations are disjoint.

\vskip .1cm
In this section we establish an analytic result, Theorem \ref{lq2b}, needed to complete the proof of Theorem \ref{localmaxO6}. 
Theorem \ref{lq2b} is a sufficient condition, applicable to the configuration $O_6$, which ensures that the point $\mathbf{0}\in\mathbb{R}^{n}$ is a sharp local maximum of the function ${\sf F}\left(  x\right)$. Theorem \ref{lq2b} is a generalization of Theorem 2, Section 5 in \cite{OS-C6}.

\subsection{Notation}\label{formsO6}
We recall some notation from \cite{OS-C6}. Till the end of the Section the summation over repeated indices is assumed.

\vskip .1cm
We denote by $l_{uj}$ and $q_{ujk}$ the coefficients of the linear and quadratic parts of the function $F_u(x)$, $u=1,\dots,m$, 
$$F_u(x)=l_{uj}x^j+q_{ujk}x^j x^k+o(2)\ ,$$
where $o(2)$ stands for higher order terms. 

\vskip .1cm
Let $\xi^j:=dx^j$, $j=1,\dots,n$, be the coordinates, corresponding to the coordinate system $x^1,\dots,x^n$, in the tangent space 
to  $\mathbb{R}^n$ at the origin. We define the linear and quadratic forms $l_u(\xi)\equiv l_{uj}\xi^j$ and $q_u(\xi)\equiv q_{ujk}\xi^j\xi^k$ on the 
tangent space  $T_0\mathbb{R}^n$.

\vskip .1cm
Let $E$ be the subspace in $T_0\mathbb{R}^n$ defined as the intersection of kernels of the linear forms $l_u(\xi)$, 
$$E=\bigcap_{u=1}^m\; \ker l_u(\xi)\ .$$

\vskip .1cm
Let $\mu=\{\mu^1\,\dots ,\mu^m\}$ be a linear dependency between the linear parts of the functions $F_u(x)$, $u=1,\dots,m$, that is, 
$$\mu^u l_{uj}=0\ \ \text{for all}\ j=1,\ldots,n\ .$$
We denote by $\mathfrak{q}[\mu ]$ the corresponding quadratic form on the space $E$ defined by
$$\mathfrak{q}[\mu ]=\mu^u q_{ujk}\xi^j\xi^k\vert_E\ .$$

\subsection{Positively defined families of quadratic forms}
We shall say that a family $\{\mathfrak{Q}_1,\ldots ,\mathfrak{Q}_L\}$ of quadratic forms on a real vector space $\mathbb{V}$
is positively defined if the following condition holds
\begin{equation}\label{mima2}\begin{array}{l}
\text{the system of inequalities}\ \ 
\mathfrak{Q}_u(x)\leq 0\ ,\ u=1,\ldots ,L,\\[.6em]
\text{admits only the trivial solution $x=0$}\ . \end{array}
\end{equation}
Also, we say that a family $\{\mathfrak{Q}_1,\ldots ,\mathfrak{Q}_L\}$ of quadratic forms on a space $\mathbb{V}$
is negatively defined if  the family $\{ -\mathfrak{Q}_1,\ldots ,-\mathfrak{Q}_L\}$ is positively defined. 

\vskip .1cm
The notion of a positively defined family of quadratic forms generalizes the notion of a positively defined quadratic form (it corresponds
to $L=1$). 

\vskip .1cm
Let
$$\mathfrak{Q}(x):=\max \bigl(\mathfrak{Q}_1(x),\ldots ,\mathfrak{Q}_L(x)\bigr)\ .$$
The condition (\ref{mima2}) is satisfied if and only if the constant 
$$\mathfrak{v}:=\min_{\left\Vert
x\right\Vert =1} \bigl(\mathfrak{Q}(x)\bigr)$$
is positive, $\mathfrak{v}>0$. Because of the homogeneity we have 
$$\mathfrak{Q}(x)\geq \mathfrak{v} \left\Vert x\right\Vert^2\ \text{for any}\ x\in\mathbb{R}^n.$$
So we can reformulate the positivity of a family in the following form. 

\begin{definition}\label{depofa}
A family $\{\mathfrak{Q}_1,\ldots ,\mathfrak{Q}_L\}$ of quadratic forms is positively defined iff there exists a positive constant $\mathfrak{v}>0$ such that for any $x\in\mathbb{R}^n$ 
\begin{equation}\label{suschac}
\text{there exists $a_\circ \in\{1,\dots ,L\}$
such that $\mathfrak{Q}_{a_\circ}(x)\geq \mathfrak{v} \left\Vert x\right\Vert^2$.}
\end{equation}
We shall say that such family $\{\mathfrak{Q}_1,\ldots ,\mathfrak{Q}_L\}$ is $\mathfrak{v}$-positively defined.
\end{definition}

As for $L=1$, the positivity of a family of quadratic forms is an open condition in the following sense. 

\begin{lemma}\label{inchava}
The condition (\ref{mima2}) is stable under small perturbations of the forms of the family. 
\end{lemma}

\begin{proof} Assume that a family $\{\mathfrak{Q}_1,\ldots ,\mathfrak{Q}_L\}$ of quadratic forms is positively defined and 
let $\mathfrak{v}$ be a constant from Definition \ref{depofa}.

\vskip .1cm 
Let $\mathfrak{P}_u$, $u=1,\ldots ,L$, be an arbitrary family of quadratic forms. There exists a positive constant $\mathfrak{w}$ such that
$$\vert\mathfrak{P}_u(x)\vert\leq \mathfrak{w} \left\Vert x\right\Vert^2\ ,\ u=1,\ldots ,L\ .$$
Given $x\in\mathbb{R}^n$, let $a_{\circ}$ be the index defined by (\ref{suschac}).  For a positive $\epsilon$ we have 
$$\vert \mathfrak{Q}_{a_{\circ}}(x)+\epsilon\,\mathfrak{P}_{a_{\circ}}(x)\vert \geq \vert \mathfrak{Q}_{a_{\circ}}(x)\vert- 
\epsilon\,\vert\mathfrak{P}_{a_{\circ}}(x)\vert
\geq \left( \mathfrak{v}-\epsilon \mathfrak{w}\right) \left\Vert x\right\Vert^2\ ,$$
therefore, the family $\{\mathfrak{Q}_u+\epsilon\,\mathfrak{P}_u\, ,\, u=1,\ldots ,L\}$ satisfies the condition of Definition \ref{depofa} for $\epsilon$
small enough.
\end{proof}

\subsection{Analytic theorem}\label{anareO6}
The particularity of the situation analyzed in Subsections \ref{secdif} and \ref{secdifb} can be described as follows. 
The family of functions $\{ F_1(x),\dots,F_m(x)\}$ splits into several subfamilies 
$$\mathcal{F}_1=\{ F_{1,1}(x),\dots,F_{1,m_1}(x)\}\ ,\ \dots\ ,\ \mathcal{F}_L=\{ F_{L,1}(x),\dots,F_{L,m_L}(x)\}$$ such that:
\begin{itemize}
\item[(A)] In each subfamily $ \mathcal{F}_a$ there is exactly one linear dependency $\lambda_a$ between the linear parts of the functions in the subfamily, 
and this dependency is convex for each subfamily.
\item[(B)] The set of variables $x^1,\ldots ,x^n$ is a union of disjoint sets $\mathcal{X}_a$, $a=1,\dots,L$, and a set $\mathcal{Y}$ with the following property: the linear parts of functions from the subfamily $\mathcal{F}_a$ depend only on the variables from the set 
$\mathcal{X}_a$ for each $a=1,\dots,L$.   
\item[(C)]The family of quadratic forms $\mathfrak{q}[\lambda_a ]$, $a=1,\dots ,L$
is negatively defined on the subspace $E$. 
\end{itemize}

\noindent We shall use the following notation for the variables from the subsets $\mathcal{X}_a$ and $\mathcal{Y}$:
$$ \mathcal{X}_a=\{ x_a^1,\dots ,x_a^{d_a}\}\ ,\ a=1,\dots,L\ ,\ \ \text{and}\ \  \mathcal{Y}=\{ y^1,\dots ,y^{d}\}\ .$$
In particular, $d_1+\ldots +d_L+d=n$. The variables $\{ y^1,\dots ,y^{d}\}$ do not enter the linear parts of functions $\{ F_1(x),\dots,F_m(x)\}$. 

\begin{lemma}\label{inchavab}
The conditions {\rm (A)} and {\rm (C)} are invariant under an arbitrary analytic change of variables, preserving the origin, such that the linear parts 
transform inside each group $\mathcal{X}_a$ of variables
\begin{equation}\label{chavar}x_a^j=(A_a)^j_k\tilde{x}_a^k+ o(1)\ .\end{equation}
Here each of matrices $A_a$, $a=1,\ldots ,L$, is non-degenerate.\end{lemma}

\begin{proof} This is a straightforward generalization of the proof of Lemma 3, Section 5, in \cite{OS-C6}.\end{proof}

\vskip .3cm
We now formulate our analytic theorem.
\begin{theorem}\label{lq2b}
Under the conditions {\rm (A)}, {\rm (B)} and {\rm (C)}, the origin is the strict local maximum of the function ${\sf F}(x)$.
\end{theorem}

\vskip .1cm
\noindent{\bf Remarks.} 

\vskip .3cm
\noindent{\bf 1.}
Our particular case of the configuration $O_6$ corresponds to $L=3$; we have the following subfamilies of functions
$$\begin{array}{c}
\mathcal{F}_1=\left\{ d^2_{\ell_1^+,\ell_2^-}\ ,\ d^2_{\ell_1^+,\ell_2^+}\ ,\ d^2_{\ell_1^-,\ell_2^-}\ ,\ d^2_{\ell_1^-,\ell_2^+}\right\}\ ,\\[1em]
\mathcal{F}_2=\left\{ d^2_{\ell_1^+,\ell_3^-}\ ,\ d^2_{\ell_1^+,\ell_3^+}\ ,\ d^2_{\ell_1^-,\ell_3^-}\ ,\ d^2_{\ell_1^-,\ell_3^+}\right\}\ ,\\[1em]
\mathcal{F}_3=\left\{ d_{\ell_{2}^{-},\ell_{3}^{-}}^{2}\ ,\ d_{\ell_{2}^{-},\ell_{3}^{+}}^{2}\ ,\ d_{\ell_{2}^{+},\ell_{3}^{-}}^{2}\ ,\ d_{\ell_{2}^{+},\ell
_{3}^{+}}^{2}\right\}\ .\end{array}$$ 
Each subfamily contains four functions. 

\vskip .1cm
The property (A) refers to formulas (\ref{depe1})--(\ref{depe3}). 

\vskip .1cm
For the property (B) see expressions (\ref{disjvar}).  

\vskip .1cm
The property (C) is the statement about three quadratic forms $\Upsilon_{1},\Upsilon
_{2},\Upsilon_{3}$ defined by formulas $\left(  \ref{55})-(\ref{57}\right)$: we have checked in Subsection \ref{secdif} that 
$\min\left(  \Upsilon_{1},\Upsilon_{2},\Upsilon_{3}\right)\,\rule[-.22cm]{0.1mm}{.54cm}_{\; E}  <0$ everywhere
except the origin.

\vskip .3cm
\noindent{\bf 2.}
Similarly to the case of the configuration $C_{\mathfrak{m}}$, it follows from the proof that the function ${\sf F}(x)$ decays quadratically at zero 
along any direction in $E$ and decays linearly along any direction outside $E$.

\vskip .1cm
As for Theorem 2, Section 5, from \cite{OS-C6}, we need the assertion of Theorem \ref{lq2b} for a family of analytic functions $F_j$; 
however, a careful analysis shows that the assertion of Theorem \ref{lq2b} holds for functions $F_j$ of the class $\mathcal{C}^{3}$.

\vskip .1cm
In \cite{OS-C6} we have given two different proofs of Theorem 2. Here we present an analogue of the first proof in \cite{OS-C6}. 
A proof of Theorem \ref{lq2b} generalizing the second proof of Theorem 2 in \cite{OS-C6} can be given as well, but it looks less natural 
and more involved, so we have decided to omit it. 

\subsection{Proof}\label{fpfO6}
We proceed as in the first proof of Theorem 2 in \cite{OS-C6}. Performing, if necessary, a suitable change of variables,  
satisfying the conditions of Lemma \ref{inchavab}, we may assume that each subfamily $ \mathcal{F}_a$, $a=1,\dots,L$, 
consists of linear functions, except for the first one,   
$$\begin{array}{c}F_{a,1}=-\sum_{i=2}^{m_a}\lambda^i_{a} x_a^i+q_a(x)+o(2)\ \ \text{where}\ \lambda^i_{a}>0\ \ \text{for}\ i=2,\dots ,m_a\ ,\\[.4em]
F_{a,2}=x_a^2\ ,\ \dots \ ,\ F_{a,m_a}=x_a^{m_a}\ .\end{array}$$
The set of variables is split into two disjoint parts, 
$$\mathcal{V}_1:=\{x_a^t\}_{t=2,\dots,m_a}^{a=1,\dots,L}\ \text{and}\ \mathcal{V}_2:=\{x_a^1\}_{a=1,\ldots ,L}\sqcup \mathcal{Y}\ .$$
We rename, for convenience, the variables from $\mathcal{V}_2$:
$ \mathcal{V}_2=\{\mathsf{y}^1,\ldots ,\mathsf{y}^{K}\}$.
We identify the points of the tangent subspace $E\subset T_0\mathbb{R}^n$ (in a small enough neighborhood $U$ of the origin) with the 
plane defined by $x_a^t=0$, $x_a^t\in \mathcal{V}_1$, and coordinatize the space $E$ by the variables
$\mathsf{y}\in\mathcal{V}_2$.

\vskip .1cm
We need to prove that in the set $E^+$, defined by the system of inequalities $x_a^t\geq 0$, $x_a^t\in \mathcal{V}_1$,
the function ${\sf{F}}^{(1)}(x)=\min_{1\leq a\leq L} (F_{a,1}(x))$ has a sharp local maximum at the origin. We make a substitution, allowed in $E^+$,
$x_a^t=(z_a^t)^2$, $x_a^t\in \mathcal{V}_1$. Our functions have the form 
$$F_{a,1}=\mathfrak{Q}_a+\psi_a\ ,\ \text{where}\ \ \mathfrak{Q}_a=-\sum_{i=2}^{m_a}\lambda^i_{a} (z_a^i)^2+\mathfrak{q}_a(\mathsf{y})\ \text{and}\ \psi_a=o(2)\ .$$
 The family $\{ \mathfrak{q}_a\}$ is negatively defined on $E$ hence the family $\{ \mathfrak{Q}_a\}$ is negatively defined on the space $\tilde{\mathbb{R}}^n$ with the coordinates
$z_a^t$ for $x_a^t\in \mathcal{V}_1$, and $\mathsf{y}\in\mathcal{V}_2$: if $\mathfrak{Q}_a\geq 0$, $1\leq a\leq L$, then  
$\mathfrak{q}_a(\mathsf{y})\geq \sum_{i=2}^{m_a}\lambda^i_{a} (z_a^i)^2\geq 0$, $1\leq a\leq L$, implying that $\mathsf{y}={\bf 0}$ and, 
consequently, $z_a^i=0$ for $1\leq a\leq L$, $2\leq i\leq m_a$.

\vskip .1cm
Since the family $\{ \mathfrak{Q}_a\}$ is negatively defined, there exists a positive constant  $\mathfrak{v}>0$ such that 
for any $x\in\tilde{\mathbb{R}}^n$ there exists $a_\circ (x) \in\{1,\dots ,L\}$ for which 
$$\mathfrak{Q}_{a_\circ (x)}(x)\leq -2\mathfrak{v} \left\Vert x\right\Vert^2\ .$$

Due to the order of smallness of functions $\psi_a$ there exists a neighborhood $U$ of the origin in $\tilde{\mathbb{R}}^n$, in which  
$$\vert\psi_a(x)\vert \leq \mathfrak{v} \left\Vert x\right\Vert^2\ \ \text{for any}\ a=1,\ldots ,L\ .$$
Therefore, for any $x\in U$ we have ${\sf{F}}_1(x)\leq F_{a_\circ (x),1}(x)\leq -\mathfrak{v} \left\Vert x\right\Vert^2$.
\myblacksquare

\vskip .4cm\noindent 
{\textbf{Acknowledgements.}
Part of the work of S. S. has been carried out in the framework of the Labex Archimede (ANR-11-LABX-0033) and of the A*MIDEX project (ANR-11-
IDEX-0001-02), funded by the Investissements d'Avenir French Government program managed by the French National Research Agency (ANR). Part of the work of S. S. has
been carried out at IITP RAS. The support of Russian Foundation for Sciences (project No. 14-50-00150) is gratefully acknowledged by S. S. The work of O. O. was supported by
the Program of Competitive Growth of Kazan Federal University and by the grant RFBR 17-01-00585.}

\end{document}